\documentclass[12pt]{article}
\usepackage{amsmath,amssymb,amsthm,latexsym,euscript,amscd, authblk,calligra, amssymb,mathrsfs}
\usepackage[english]{babel}

\usepackage{tikz-cd} 
\usepackage{adjustbox}
\usepackage{quiver}
\usepackage{stackrel}
\usepackage[all]{xy}
\RequirePackage{graphicx}
\usepackage{tikz}
\usepackage{easybmat}

\usepackage{xcolor}
\usepackage{thmtools,hyperref,cleveref}
\DeclareMathOperator{\Hom}{Hom}
\DeclareMathOperator{\Ext}{Ext}

\DeclareMathOperator{\Spec}{Spec}

\renewcommand{\P}{\mathbb{P}}
\renewcommand{\phi}{\varphi}

\newcommand{\id}{\mathrm{id}}

\newcommand{\F}{F}
\newcommand{\G}{G}
\newcommand{\Perv}{\mathrm{Perv}}
\newcommand{\Desc}{\mathrm{Desc}}

\newcommand{\Cons}{\mathrm{Cons}}

\DeclareMathOperator{\sheafhom}{\mathscr{H}\text{\kern -5.5pt {\calligra\large om}}\,}
\newenvironment{psmallmatrix}
  {\left(\begin{smallmatrix}}
  {\end{smallmatrix}\right)}
\newtheorem{Theorem}{Theorem}
\newtheorem{Proposition}[Theorem]{Proposition}
\newtheorem{Corollary}[Theorem]{Corollary}

\newtheorem{Definition}[Theorem]{Definition}

\newtheorem{Remark}[Theorem]{Remark}
\theoremstyle{definition}
\newtheorem{Example}[Theorem]{Example}
\newtheoremstyle{case}{}{}{}{}{}{:}{ }{}
\theoremstyle{case}

\newcommand{\C}{{\mathbb{C}}}
\newcommand{\K}{{K}}

\title{Perverse sheaves on smooth toric varieties and stacks}

\author{Sergey Guminov}

\affil{National Research University Higher School of Economics\\

Centre of Pure Mathematics, MIPT}

\usepackage{tikz}
\usetikzlibrary[patterns]

\numberwithin{Proposition}{section}
\numberwithin{Corollary}{section}
\numberwithin{Lemma}{section}
\numberwithin{Example}{section}
\numberwithin{Conjecture}{section}
\numberwithin{Question}{section}
\numberwithin{Remark}{section}
\numberwithin{Theorem}{section}
\numberwithin{Definition}{section}

\begin{document}


\maketitle
\begin{abstract}
    It is usually not straightforward to work with the category of perverse sheaves on a variety using only its definition as a heart of a $t$-structure. In this paper, the category of perverse sheaves on a smooth toric variety with its orbit stratification is described explicitly as a category of finite-dimensional modules over an algebra. An analogous result is also established for various categories of equivariant perverse sheaves, which in particular gives a description of perverse sheaves on toric orbifolds, and we also compare the derived category of the category of perverse sheaves to the derived category of constructible sheaves.
\end{abstract}

\footnotetext[0]{The research was funded by the Russian Science Foundation (project № 25-21-00358).}
\section{Introduction}
Perverse sheaves are one of the central objects of study in topology and geometry. There are multiple approaches to describing the category of perverse sheaves smooth along a fixed stratification of a complex variety (or a more general topological space). R. MacPherson and K. Vilonen explain how to compute this category by an inductive process, gluing the strata one by one \cite{MacPherson1986inventiones}, and, with S. Gelfand, how to express this category using microlocal data \cite{gelfand2005microlocal}; A. Beilinson's gluing theorem uses nearby and vanishing cycles to glue the whole category from perverse sheaves on a divisor and its complement \cite{Beilinson1987glue}. Still, the categories of perverse sheaves have only been described explicitly for a small number of families of stratified varieties. Among those are affine spaces stratified by hyperplane arrangements \cite{Kapranov2016hyperplane}, spaces of matrices stratified by the rank \cite{braden1999rank}, Grassmanians of types A and D with the Schubert stratification \cite{braden2002grassmanian}, affine spaces with the normal crossings stratification \cite{galligo1985mathcal} and, more generally, smooth toric varieties \cite{Dupont2010thesis}. 

This paper is dedicated to the study of the category of (equivariant) perverse sheaves and its derived category on a smooth toric variety. It is possible to explicitly describe the category of perverse sheaves as the category of finite-dimensional modules over an algebra which is a finite extension of the group algebra of the fundamental group of the torus. For a variety of dimension one this observation is due to A. Bondal and T. Logvinenko \cite{bondal_logvinenko}. The quiver description of this category was obtained by D. Dupont in her thesis \cite{Dupont2010thesis} by gluing the quiver descriptions of perverse sheaves on an affine cover by open subvarieties of the form $\mathbb{C}^n\times(\mathbb{C}^\times)^m$, the perverse sheaves on which were already described earlier in \cite{galligo1985mathcal}. Unfortunately, the description of the glued category turned out to be somewhat incorrect. Following the same general idea, we present a corrected version of this description. The language of monadic descent makes it possible to use the stacky nature of perverse sheaves while mostly avoiding 2-categorical difficulties.

The category of $G$-equivariant perverse sheaves, for $G$ a closed subgroup of the torus, is also described. Combined with the Cox construction, this also gives an alternative proof of the result for the non-equivariant sheaves mentioned earlier. We still present the argument using monadic descent because it is expected to generalize to the non-smooth case.

There is also the question of whether the derived category of the abelian category of perverse sheaves is equivalent to the derived category of constructible sheaves (i.e. complexes with constructible cohomology) under the realization functor. When the stratification is not fixed a priori, this is the content of Beilinson's theorem \cite{beilinson2006derived}. For a fixed stratification, we describe a class of stratifications $\Sigma$ for which the realization functor is an equivalence for the category of perverse sheaves constructible with respect to $\Sigma$. In particular, this class includes the orbit stratification of a normal toric variety. 

The author is a winner of the all-Russia mathematical August Moebius contest of graduate and undergraduate student papers and thanks the jury and the board for the high praise of his work. The author is also grateful to Alexey Bondal for extensive comments on an earlier draft of this work and to Sasha Beilinson for illuminating discussions.

\subsection{Preliminaries and notation}
Throughout the paper, the term "toric variety" means a normal complex toric variety with its analytic topology, and the word "cone" will mean a strongly convex rational polyhedral cone. Sometimes for a cone $\sigma$ we also let $\sigma$ denote the fan of all faces of $\sigma$. $\tau\prec\sigma$ will mean that $\tau$ is a face of $\sigma$.  We will be working exclusively with sheaves with values in $k$-vector spaces over some field $k$. By "perverse sheaf" we mean a middle perversity perverse sheaf, with the indexing convention such that on a smooth variety $X$ of dimension $n$ the constant sheaf $k_X$ shifted by $n$, i.e. the object $k_X[n]$, is perverse. When working with perverse sheaves on toric varieties, we will always assume the perverse sheaves to be constructible with respect to the orbit stratification. The category of perverse sheaves on a variety $X$ with a stratification $\Sigma$ is denoted by $\Perv(X)$ or $\Perv_\Sigma(X)$. Also, all the direct and inverse image functors will always be derived by default and denoted simply by, for example, $j_*$ instead of $Rj_*$. When working with fundamental groups, we will always assume that a base point is fixed and simply write $\pi_1(X)$.

For the general theory of perverse sheaves we refer to the original paper \cite{BBD1982} or the book \cite{achar2021perverse}.


\section{Perverse sheaves on a smooth toric variety}

In this section our goal is to understand the category of perverse sheaves on a smooth toric variety with its orbit stratification as a category of finite-dimensional modules over some algebra. The starting point is the case of the normal crossings stratification of an affine space $\mathbb{C}^n.$ The case $n=1$ is by now classic: the data of a perverse sheaf on $\mathbb{C}$ stratified by the origin and its complement is equivalent to the data of a pair of $k$-vector spaces
\[\begin{tikzcd}
	\Psi && \Phi
	\arrow["u", curve={height=-6pt}, from=1-1, to=1-3]
	\arrow["v", curve={height=-6pt}, from=1-3, to=1-1]
\end{tikzcd}\] such that $\id_\Psi-vu$ and $\id_\Phi-uv$ are invertible (in fact, it is sufficient to require this for only one of these maps) \cite{verdier1985extension,Beilinson1987glue}. The equivalence between the two categories is given by mapping a perverse sheaf $F$ to the diagram 
\[\begin{tikzcd}
	{\Psi(F)} & {} & {\Phi(F)}
	\arrow["{\mathrm{can}}", curve={height=-6pt}, from=1-1, to=1-3]
	\arrow["{\mathrm{var}}", curve={height=-6pt}, from=1-3, to=1-1]
\end{tikzcd}\] where $\Psi$ and $\Phi$ denote the nearby and vanishing cycles functors (possibly shifted, depending on conventions) and $\mathrm{var},\ \mathrm{can}$ are the natural maps between them. In this case $\Psi(F)$ with the monodromy operator $\id_{\Psi(F)}-\mathrm{var}\circ \mathrm{can}$ may be viewed as a local system on $\mathbb{C}^\times$ and then it will coincide with the local system $F|_{\C^\times}[-1].$ 

The case $n>1$ was resolved by Maisonobe, Galligo and Granger \cite{galligo1985mathcal}. In their description, a perverse sheaf on $\C^n$ is identified with a collection of $2^n$ vector spaces $V_i$, $i\in \{0,1\}^n$, one for each stratum, and maps between them. If we think of the spaces as placed in the vertices of an $n$-dimensional hypercube, then there is a pair of maps $u_e$ and $v_e$ in opposite directions at each edge $e$, and these maps are such that all squares commute, the maps $\id-u_e\circ v_e$ and $\id-v_e\circ u_e$ are invertible, and all such automorphisms commute with each other. 

A smooth toric variety is glued from affine charts of the form $\C^n\times (\C^\times)^m.$ The category of perverse sheaves on smooth toric varieties was described by Dupont \cite{Dupont2010thesis} by gluing the descriptions of the above form, using the fact that perverse sheaves satisfy smooth descent. For the case of $\mathbb{P}^2$ considered as a smooth toric variety, the category of perverse sheaves is described as the category of diagrams of the form 
\[\begin{tikzcd}[sep=scriptsize]
	&&&& {V_1} \\
	\\
	{V_{13}} &&&&&&&& {V_{12}} \\
	\\
	&&&& V \\
	\\
	{V_3} &&&&&&&& {V_2} \\
	\\
	&&&& {V_{23}}
	\arrow["{u_1}", curve={height=-12pt}, from=5-5, to=1-5]
	\arrow["{v_1}", curve={height=-12pt}, from=1-5, to=5-5]
	\arrow["{u_{12}}", curve={height=-12pt}, from=1-5, to=3-9]
	\arrow["{v_{12}}", curve={height=-12pt}, from=3-9, to=1-5]
	\arrow["{v_{21}}", curve={height=-12pt}, from=3-9, to=7-9]
	\arrow["{u_{21}}", curve={height=-12pt}, from=7-9, to=3-9]
	\arrow["{u_2}", curve={height=-12pt}, from=5-5, to=7-9]
	\arrow["{v_2}", curve={height=-12pt}, from=7-9, to=5-5]
	\arrow["{u_3}", curve={height=-12pt}, from=5-5, to=7-1]
	\arrow["{u_{23}}", curve={height=-12pt}, from=7-9, to=9-5]
	\arrow["{v_{23}}", curve={height=-12pt}, from=9-5, to=7-9]
	\arrow["{v_{32}}", curve={height=-12pt}, from=9-5, to=7-1]
	\arrow["{u_{31}}", curve={height=-12pt}, from=7-1, to=3-1]
	\arrow["{v_{31}}", curve={height=-12pt}, from=3-1, to=7-1]
	\arrow["{v_{13}}", curve={height=-12pt}, from=3-1, to=1-5]
	\arrow["{u_{13}}", curve={height=-12pt}, from=1-5, to=3-1]
	\arrow["{u_{32}}", curve={height=-12pt}, from=7-1, to=9-5]
	\arrow["{v_3}"', curve={height=-12pt}, from=7-1, to=5-5]
\end{tikzcd}\] such that all squares commute, the mappings $M_i=v_iu_i+\id$, $M_{ij}=v_{ij}u_{ij}+\id$, $N_i=u_iv_i+\id$, $N_{ij}=u_{ij}v_{ij}+\id$ are invertible and commute (and the $+$ sign instead of $-$ is, of course, a matter of convention only), and such that $M_3=M_1^{-1}M_2^{-1}$ and $M_{ij}=M_{ik}^{-1}$ for all $(i,j,k)\in\{1,2,3\}^3.$ 

A closer look reveals a contradiction. Using the above relations, we find

$$v_1=v_1M_{13}M_{12}=M_3M_2v_1=M_1^{-1}v_1.$$ A diagram as above is obtained by gluing the 3 square diagrams corresponding to the three charts isomorphic to $\C^2$, and the restriction functor to such a chart corresponds to taking a square subdiagram. The relation $v_1=M_1^{-1}v_1$ deduced above implies that the functor of taking the subdiagram spanned by $V,\ V_1,\ V_2,\ V_{12}$ is not essentially surjective. However, the embedding $j: {\mathbb{C}}^2\hookrightarrow\P^2$ is affine, so both $j_!$ and $j_*$ must preserve perverse sheaves, and hence every $F\in\Perv(\mathbb{C}^2)$ is isomorphic to $j^*j_* F$ and lies in the image of $j^*$. 

To get the correct relations, one should think of the strata of the toric stratification as given by intersecting the torus-invariant divisors $D_i$. Given a diagram of vector spaces as above, consider an adjacent pair of vector spaces $V,\ W$ and the maps $u,\ v$ between them. Such a pair corresponds to a pair of adjacent strata $S$ and $S'$, i.e. such that $S'$ is of codimension one in the closure $\overline{S}$. Then $S'$ is locally cut out in $\overline{S}$ by intersecting with some divisor $D_i$, and the invertible operators $M=vu+\id$, $N=uv+\id$ should be thought of as monodromy operators corresponding to a loop in the open orbit $T$ of the toric variety which go around $D_i$ once. The relations between the various operators then will follow from the relations between such loops in $\pi_1(T).$ In particular, in the case of $\mathbb{P}^2$ described above the relations $M_{ij}=M_{ik}^{-1}$ must be replaced with $M_{ij}=M_{ik}^{-1}N_i^{-1}.$ This way of thinking about these relations seems to have already been known to some experts in mirror symmetry \cite{gammage2023homological}.

In this section, our goal is to prove the corrected version of the above description of the category of perverse sheaves on smooth toric varieties. One distinction is that instead of gluing diagrams of vector spaces, we first identify perverse sheaves on affine varieties with an action of a torus $T$ with modules over algebras which are finite extensions of $k[\pi_1(T)]$, and then glue these categories of modules into a category of modules over an even larger algebra, explicitly defined by the fan of the toric variety. 

Note that the description of the category of equivariant perverse sheaves presented in the next section will also give a proof for the nonequivariant case through the use of the Cox construction. Nevertheless, we choose to present the "stacky" proof anyway. One reason for that is that it is likely that if one could somehow identify $\Perv(X)$ on a non-smooth affine toric variety $X$ acted on by a torus $T$ with the category of modules over some extension of $k[\pi_1(T)]$, then the argument presented in this section would help with gluing these local descriptions into a description of the category $\Perv(X)$ for an arbitrary toric $X$.

Fix an algebraic torus $T=(\mathbb{C}^\times)^n.$ Consider its cocharacter lattice $N$ (also known as the lattice of one-parameter subgroups) and a regular fan $\Delta\subset N\otimes \mathbb{Q}$ consisting of $q$ cones, $m$ of which are one-dimensional, which corresponds to a smooth toric $n$-dimensional variety $X_\Delta$. The torus $T$, which we may now identify with $N\otimes {\mathbb{C}}^\times$, acts on $X_\Delta$. The fundamental group $\pi_1(T)$ is naturally identified with the lattice $N$ by the isomorphism $\lambda: N\to \pi_1(T)$ taking a morphism $\mathbb{C}^\times \to T$ to its restriction to $S^1\subset \mathbb{C}^\times$. Any ray of the fan $\Delta$ is generated by an integral element $v_i\in N,$ where $i=1,\ldots,m$, and we denote the corresponding loop in $\pi_1(T)$ by $\lambda_i=\lambda(v_i).$

Consider the algebra of matrices $\mathrm{Mat}_q(k[N])$ of size $q\times q$. We will be indexing the rows and columns of a matrix by the cones of the fan and let $M_{\sigma,\tau}$ denote the matrix element of $M$ in row $\sigma$ and column $\tau$ for $\sigma,\tau$ -- some cones of $\Delta$. Finally, define $$A(\Delta)=\{M\in\mathrm{Mat}_q(k[N])\ |\ M_{\sigma,\tau} \text{ is divisible by }\prod_{i: v_i\in\sigma\setminus\tau}(1-v_i)\quad \forall \sigma,\tau\}.$$ It is easy to see that this is indeed an algebra and that its center consists precisely of the scalar matrices. 

There is a more geometric definition of this algebra. Consider the dual torus $T^\vee=\Spec k[N]$ and a trivial bundle $E$ of rank $q$ on it with the basis of the space of sections $\varepsilon_\sigma$, indexed by the cones $\sigma$ of the fan $\Delta$. Then $A(\Delta)$ is the algebra of endomorphisms of $E$ that preserve over $V(1-v_i)$ the subbundles generated by $\varepsilon_\sigma$ such that $v_i\notin\sigma.$ 

It is also of note that since $k[\pi_1(T)]$ is naturally identified with the center of $A(\Delta)$, and hence $A(\Delta)$ may be considered as a coherent sheaf on $T^\vee$. The fiber of $A(\Delta)$ over a general point is just a matrix algebra. The fibers over the points of the divisors $V(1-v_i)$ are more interesting. These divisors are important because it will turn out that the fiber of a module over $A(\Delta)$ which corresponds to a perverse sheaf $F$ at a point of a divisor $V(1-v_i)$ is nontrivial precisely when the loop $\lambda_i$ acts on the local system $F|_T$ in a unipotent way, and it is precisely such local systems that have nontrivial extensions to a perverse sheaf along $D_i$.

For explicit computations, we assume that we have already chosen a basis $\varepsilon_1,\ldots,\varepsilon_n$ of $N$ and hence have $T^\vee=(\C^\times)^n=\Spec k[t_1^{\pm 1},\ldots,t_n^{\pm 1}]$, with $t_i=\lambda(\varepsilon_i).$
\begin{Example}\label{Example_C1}

The category $\Perv(\mathbb{C})$ of perverse sheaves on $\mathbb{C}$ stratified by 0 and its complement is equivalent to the category of diagrams 
\[\begin{tikzcd}
	\Psi && \Phi
	\arrow["u", curve={height=-6pt}, from=1-1, to=1-3]
	\arrow["v", curve={height=-6pt}, from=1-3, to=1-1]
\end{tikzcd}\] such that $\id_\Psi-vu$ and $\id_\Phi-uv$ are invertible. As long as we forget about the invertibilty conditions for the time being, these diagrams are just representations of a quiver $Q$ with two arrows and two vertices. Consider the path algebra $kQ$ with the path of length zero labeled $e_1,e_2$. Then the invertibilty conditions above are equivalent to the invertibility of the algebra element $s=1+vu+uv$. This element is central, so the localization $kQ_s$ of $kQ$ with respect to $s$ may be defined as usual by inverting $s$, and in fact, we get an equivalence between $\Perv(\C^1)$ and $kQ_s\text{-}\mathrm{mod^{fd}}$. Moreover, if under this equivalence a perverse sheaf $F$ corresponds to a $kQ_s$-module $M$, then the restriction of $F$ to the open stratum is the local system corresponding to the $e_1kQ_se_1\cong k[t,t^{-1}]$-module $e_1M$, and the action of $t=\id_1-vu$ corresponds to going around the loop around the origin.

Let $\Delta$ be the fan in $\mathbb{Q}^1$ consisting of 0 and the cone generated by $1$. Then the algebra $kQ_s$ is isomorphic to the algebra $A(\Delta)=\{\begin{psmallmatrix}a & b\\(1-t)c & d\end{psmallmatrix}\in\mathrm{Mat}_2(k[t,t^{-1}])\ |\ a,b,c,d\in k[t,t^{-1}] \}$ by the isomorphism defined on generators by the rule

$$e_1\mapsto \begin{psmallmatrix}1 & 0\\0 & 0\end{psmallmatrix},\ e_2\mapsto \begin{psmallmatrix}0 & 0\\0 & 1\end{psmallmatrix},\ u\mapsto \begin{psmallmatrix}0 & 0\\ 1-t & 0\end{psmallmatrix},\  v\mapsto \begin{psmallmatrix}0 & 1\\ 0 & 0\end{psmallmatrix},$$ so the idempotent $e_1$ corresponds to the zero cone and $e_2$ corresponds to the ray. Geometrically speaking, $A(\Delta)$ is the algebra of endomorphisms of a trivial rank $2$ bundle over $\C^\times$ preserving a line in the fiber over $1$. The description of the category $\Perv(\mathbb{C})$ as a category of modules over this algebra previously appeared in the paper \cite{bondal_logvinenko}.

\end{Example}

\begin{Example}\label{Example_P^1}
For $X_\Delta=\P^1$ the fan $\Delta$ consists of three cones $\{0\}, \sigma, -\sigma$ generated by $0, 1$ and $-1$ respectively, and the algebra $A(\Delta)$ is the algebra of endomorphisms of the trivial bundle of rank 3, preserving a pair of planes intersecting in a line in the fiber over $1\in\mathbb{C}^\times$. In terms of diagrams, the category of perverse sheaves corresponds to diagrams of the form 
\[\begin{tikzcd}
	{\Phi_1} && \Psi && {\Phi_2}
	\arrow["{u_2}", curve={height=-6pt}, from=1-3, to=1-5]
	\arrow["{v_2}", curve={height=-6pt}, from=1-5, to=1-3]
	\arrow["{u_1}"', curve={height=6pt}, from=1-3, to=1-1]
	\arrow["{v_1}"', curve={height=6pt}, from=1-1, to=1-3]
\end{tikzcd}\] with the already familiar invertibilty conditions and the additional relation $\id_\Psi-v_1u_1=(\id_\Psi-v_2u_2)^{-1}.$ As a matrix algebra, $A(\Delta)$ consists of matrices of the form

\[\begin{pmatrix}
    a & b & c\\
    (1-t)d & e & (1-t)f\\
    (1-t^{-1})g & (1-t^{-1})h & l
\end{pmatrix},\] where $a,\ldots, l$ are arbitrary Laurent polynomials and the rows and columns correspond to the cones in the order $\{0\}, \sigma, -\sigma$. Unlike the previous affine case, there doesn't seem to be a direct way to get this algebra as a localization of some (bound) path algebra of the quiver on three vertices. 
\end{Example}

The following theorem is the main result of this section.
\begin{Theorem}\label{Theorem_perverse_modules}
    The category $\Perv(X_\Delta)$ of perverse sheaves on a smooth toric variety $X_\Delta$ constructible with respect to its orbit stratification is equivalent to the category $A(\Delta)-\mathrm{mod^{fd}}$ of finite-dimensional left $A(\Delta)$-modules.
\end{Theorem}

We proceed by first identifying the category of $A(\Delta)$-modules with the category of descent data with respect to an affine cover of $X_\Delta$ by affine toric varieties and then identifying the category $A(\Delta)$ for affine $X_\Delta$ with perverse sheaves on $X_\Delta$.

\subsection{Descent data}

In this section, we are going to perform some groundwork to then apply classic results of monadic descent, namely the Beck monadicity criterion (also known under the name Barr-Beck criterion) and the Bénabou-Roubaud theorem. For an overview of monadic descent we refer the reader to \cite{janelidze1994facets}, and for the proof of the Bénabou-Roubaud theorem to \cite{kahn2024benabouroubaud}.

For a non-unital morphism of unital rings $f: B\to A$ there is a pair of adjoint functors $f_!(-)=Af(1)\otimes_B (-)$ and $f^*(-)=f(1)(-)$ between the categories of unital modules $B\text{-}\mathrm{Mod}$ and $A\text{-}\mathrm{Mod}.$  For an idempotent $e\in A$ and the inclusion $i: eAe\to A$ these simplify to $i_!(M)=Ae\otimes_{eAe} M$ and $i^*(N)=eN.$ If a pair of idempotents $e,e'\in A$ is such that $ee'=e$, then there is also an inclusion $eAe\subset e'Ae'.$ Our goal now is to construct a descent theory for modules over $A$ in which the role of open charts is performed by the subalgebras of the form $eAe$.

Consider a unital $k$-algebra $A$ and define the indexing category $\mathcal{B}$ as follows. Its objects are finite collections $e_J=(e_j)_{j\in J}$ of commuting idempotents $e_j\in A$, and the morphisms $e_J\to e'_K$ are the functions $t: J\to K$ such that  $e_je'_{t(j)}=e_j$ for all $j\in J$. The composition of morphisms is done by simply composing the functions. We use the indexing category $\mathcal{B}$ to define similar functors between products of module categories.

Consider the category $\mathcal{E}$, the objects of which are collections of pairs $(e_j, M_j)_{j\in J}$, which will henceforth be simply denoted $(e_J,M_J)$, where $e_J\in \mathcal{B}$ and $M_j\in e_jAe_j\text{-}\mathrm{mod}$. A morphism $(e_J, M_J)\to (e'_K, N_K)$ is a tuple $(t,f_J)=(t, (f_j)_{j\in J})$, where $t\in \Hom_\mathcal{B}(e_J,e'_K)$ and $f_j\in \Hom_{e_jAe_j}(M_j,e_jN_{t(j)}).$

There is an evident projection functor $P:\mathcal{E}\to\mathcal{B}.$ For $B\in \mathcal{B}$ we denote by $\mathcal{E}(B)$ the fiber of $P$ over $B$, which is the non-full subcategory of $\mathcal{E}$ consisting of objects $E\in \mathcal{E}$ such that $P(E)=B$ and morphisms $f: E_1\to E_2$ such that $P(E_1)=P(E_2)=B$ and $P(f)=\id_B$. In our case, $\mathcal{E}(e_J)$ is the category $\prod\limits_{j\in J}e_jAe_j\text{-}\mathrm{mod}.$

 A reminder: a bifibration is a functor $P:\mathcal{C}\to\mathcal{D}$, which under the Grothendieck construction corresponds to an adjoint pair of pseudofunctors $\mathcal{D}\to\mathcal{C}at$. Equivalently, $P$ is a bifibration if every morphism $f$ in $\mathcal{D}$ admits both a cartesian and a cocartesian lifting. A classic example of a bifibration is the "large"\ category of modules $\mathcal{M}od$, the objects of which are tuples $(R,M)$ of a ring and a module over it, with the projection functor $\mathcal{M}od\to\mathcal{R}ing$.

The functor $P$ makes $\mathcal{E}$ into a bifibration over $\mathcal{B}$. For a morphism $p: e_J\to e'_K$ in $\mathcal{B}$ the cartesian lifting of $p$ at $X=(e'_K, N_K)$ is the pair $(p^*X,\nu_{p,X} X: p^*X\to X),$ where $p^*X=(e_J, p^*N_J)$ with $p^*N_j=e_jN_{p(j)}$, and $\nu_{p,X}$ is the tuple consisting of $p: e_J\to e'_K$ and the identity morphisms $\id: e_{j}N_{p(j)}\to e_{j}N_{p(j)}$.

The cocartesian lifting of $p$ at $Y=(e_J, M_J)$ is the pair $(p_!Y,\delta_{p,Y} Y:Y\to p_! Y),$ where $p_! Y=(e'_K, p_! M_K)$ with $p_!M_k=\prod\limits_{j: p(j)=k} e'_kAe_j\otimes_{e_jAe_j} M_j$, and $\delta_{p,Y}$ consists of $p: e_J\to e'_K$ and the morphisms $(\delta_{p,Y})_j: M_j\to  e_jp_!M_{p(j)}$ obtained by adjunction from $\id: e'_{p(j)}Ae_j\otimes_{e_jAe_j} M_j\to e'_{p(j)}Ae_j\otimes_{e_jAe_j} M_j$ followed by inclusion of a factor into the (finite) product.

The easy check that these liftings are actually (co-)cartesian is left to the interested reader. Defining $p^*, p_!$ on morphisms in the evident way, we obtain for each $p:e_J\to e'_K$ the pair of adjoint functors $p_!\dashv p_*$ between the categories $\mathcal{E}(e_J)=\prod\limits_{j\in J}e_jAe_j\text{-}\mathrm{mod}$ and $\mathcal{E}(e'_K)=\prod\limits_{k\in K}e'_kAe'_k\text{-}\mathrm{mod},$ which are generalizations of $f_!$ and $f^*$ for a morphism of rings $f: B\to A$.

Consider the morphism $p: E\to \ast$, where $E=e_J$ is arbitrary and $\ast$ is the terminal object of $\mathcal{B}$ consisting of a single idempotent $1$. We also have $\mathcal{E}(\ast)=A\text{-}\mathrm{mod}$.

\begin{Proposition}\label{monadicity}
    If $E=e_J$ is such that the ideal $(e_J)$ is equal to $A$, then the pullback functor $p^*$ along $p: E\to \ast$ is monadic.
\end{Proposition}
\begin{proof}
    By Beck's Monadicity Theorem, a functor $U:\mathcal{C}\to\mathcal{D}$ is monadic if and only if it has a left adjoint, reflects isomorphisms, and preserves certain coequalizers in $\mathcal{C}$. In our case, $p^*: A\text{-}\mathrm{mod}\to \prod\limits_{j\in J} A(E_j)\text{-}\mathrm{mod}$ is the right adjoint to $p_!$ and is a finite product of exact functors between abelian categories, so is itself exact. So it only remains to check that $p^*$ reflects isomorphisms.

    Consider a morphism of $A$-modules $f: M\to N$ and assume that $\forall j\in J$ $(p^*f)_j: e_jM\to e_jN$ is an isomorphism. Then $f$ is injective, since $f(m)=0$ would imply $f(e_j m)=0$, so $(p^*f)_j(e_jm)=0$ $\forall j\in \Delta$. But not all $e_jm$ can simultaneosly be zero, since that would imply that $(e_J)$ annihilates $m$. $f$ is also surjective, since it is surjective onto the subspaces $e_jN$, which generate $N$ as a module.
\end{proof}

The above proposition allows us to identify the category of modules over $A$ with some category of algebras (also known as modules) over a monad. To pass from the category of algebras over a monad to descent data we need a bit more. Consider the fiber square in $\mathcal{B}$
\[\begin{tikzcd}[ampersand replacement=\&,cramped]
	{E\times_B E} \&\& {E} \\
	\\
	{E} \&\& \ast
	\arrow["{\pi_1}"', from=1-1, to=3-1]
	\arrow["{\pi_2}", from=1-1, to=1-3]
	\arrow["p"', from=3-1, to=3-3]
	\arrow["p", from=1-3, to=3-3]
\end{tikzcd}\]

Explicitly, the product $E\times E$ is the set $\{e_{ij}\}_{i,j\in J}$ where $e_{ij}=e_ie_j$, and $\pi_1(e_{ij})=e_i,\ \pi_2(e_{ij})=e_j$. We also obtain the diagram of functors

\[\begin{tikzcd}
	\mathcal{E}({E\times_B E}) && \mathcal{E}({E}) \\
	\\
	\mathcal{E}({E}) && \mathcal{E}(B)
	\arrow["{\pi_{1!}}"', from=1-1, to=3-1]
	\arrow["{\pi_{2}^*}"', from=1-3, to=1-1]
	\arrow["p_!"'', from=1-3, to=3-3]
	\arrow["p^*", from=3-3, to=3-1]
\end{tikzcd}\] and a "base change" natural transformation $\beta_p: \pi_{1!}\pi_2^*\to p^*p_!$. $P:\mathcal{E}\to \mathcal{B}$ is said to satisfy the Beck-Chevalley condition at $p$ if $\beta_p$ is a natural isomorphism.  In our situation, if we take a collection $M_J$ of modules $M_j$ over $e_jAe_j$, we get for $j\in J$

$$(p^*p_!M_J)_j=\prod\limits_{i\in J} e_jAe_i\underset{e_iAe_i}{\otimes}M_i,$$

$$(\pi_{1!}\pi_2^* M_J)_j=\prod\limits_{i\in J} e_jAe_{ij}\underset{e_{ij}Ae_{ij}}{\otimes} e_{ij} M_i.$$  The natural transformation $\beta_p$ is constructed as the composite of adjunction morphisms

$$\pi_{1!}\pi_2^*\xrightarrow[]{\eta}\pi_{1!}\pi_2^*p^*p_!=\pi_{1!}\pi_1^*p^*p_!\xrightarrow[]{\varepsilon}p^*p_!.$$  From the construction of the functors involved, it is clear that this morphism is a product of morphisms $\beta_{ij}$ coming from diagrams of algebras of the form
\[\begin{tikzcd}
	e_{ij}Ae_{ij} && e_{j}Ae_{j} \\
	\\
	 e_{i}Ae_{i} && A
	\arrow["{q_{1}}", from=1-1, to=3-1]
	\arrow["{q_{2}}", from=1-1, to=1-3]
	\arrow["p_{2}"', from=1-3, to=3-3]
	\arrow["p_{1}", from=3-1, to=3-3]
\end{tikzcd}\] with the morphisms $p_i$ and $q_i$ being the embeddings.

Explicitly, $\beta_{ij}$ is the composition $$q_{1!}q_2^*\xrightarrow[]{\eta}q_{1!}q_2^*p_2^*p_{2!}=q_{1!}q_1^*p_1^*p_{2!}\xrightarrow[]{\varepsilon}p_1^*p_{2!}.$$ The unit of the adjunction $\eta: \id\to p_2^*p_{2!}$ is an isomorphism, which on each object $M$ takes the form $M\cong e_jAe_j\otimes_{e_jAe_j} M=p_2^*p_{2!}(M)$. However, the counit $\varepsilon_M: q_{1!}q_{1}^*M\to M$ is the map taking $a\otimes m\in e_iAe_{ij}\otimes_{e_{ij}Ae_{ij}}e_{ij}M$ to $am\in M$, and is not in general an isomorphism, even when evaluated on objects of the form $p_1^*p_{2!}N.$ 

Now we specialize to the earlier setting, where we had a regular fan $\Delta$ and the corresponding algebra of matrices $A(\Delta)$, which we for now denote by simply $A$.  Let $E_{\sigma,\tau}$ be the matrix with only one nonzero element 1 in position $(\sigma,\tau)$. For each cone $\sigma$ we get an idempotent $E_{\sigma,\sigma}\in A(\Delta)$ such that $1=\sum\limits_{\sigma\in\Delta} E_{\sigma,\sigma}$. For each subfan $\Delta'\in\Delta$ we have an idempotent $\sum\limits_{\tau\in\Delta'} E_{\tau,\tau}.$ In particular, we get the idempotents $e_\sigma=\sum\limits_{\tau\prec sigma} E_{\tau,\tau}.$ Note that $e_{\sigma\cap\tau}=e_\sigma e_\tau.$ We also denote $A_\sigma=e_\sigma Ae_\sigma$

 We apply the previous general construction of the bifibration $P:\mathcal{E}\to \mathcal{B}$ to the case where $A=A(\Delta)$.

\begin{Proposition}\label{base_change}
    For the algebra $A=A(\Delta)$ the bifibration $P:\mathcal{E}\to \mathcal{B}$ satisfies the Beck-Chevalley condition at $p: E=\{e_\sigma\}_{\sigma\in\Delta}\to \ast$.
\end{Proposition}
\begin{proof}

We have just seen that to see that the base change morphism is an isomorphism it is enough to check that for its components, which are the base change morphisms for the diagrams 
\[\begin{tikzcd}
	A_{\sigma\cap\tau} && A_{\tau} \\
	\\
	 A_{\sigma} && A
	\arrow["{q_{1}}", from=1-1, to=3-1]
	\arrow["{q_{2}}", from=1-1, to=1-3]
	\arrow["p_{2}"', from=1-3, to=3-3]
	\arrow["p_{1}", from=3-1, to=3-3]
\end{tikzcd}\] which are the compositions $$q_{1!}q_2^*\xrightarrow[]{\eta}q_{1!}q_2^*p_2^*p_{2!}=q_{1!}q_1^*p_1^*p_{2!}\xrightarrow[]{\varepsilon}p_1^*p_{2!}.$$ We have to show that for any $M\in e_{\tau}Ae_{\tau}\text{-}\mathrm{mod}$ the morphism $\varepsilon_{p_1^*p_{2!}M}$ is an isomorphism. Explicitly, 
$$p_1^*p_{2!}M=e_\sigma Ae_\tau\underset{A_\tau}{\otimes} M,$$
$$q_{1!}q_1^*p_1^*p_{2!}M=e_\sigma Ae_{\sigma\cap\tau}\underset{A_{\sigma\cap\tau}}{\otimes}e_{\sigma\cap\tau}Ae_\tau\underset{A_\tau}{\otimes} M,$$ and $\varepsilon_{p_1^*p_{2!}M}(a\otimes b\otimes m)=ab\otimes m.$ This morphism can be viewed as the result of applying $(-)\otimes_{e_\tau A e_\tau} M$ to the multiplication morphism $$\mu: e_\sigma Ae_{\sigma\cap\tau}\underset{A_{\sigma\cap\tau}}{\otimes}e_{\sigma\cap\tau}Ae_\tau\to e_\sigma Ae_\tau,$$ and so it is enough to show that $\mu$ itself is an isomorphism.

As a matrix algebra, $A$ consists of the matrices $Q$ such that, when rows and columns are indexed by the cones of $\Delta$, the component $Q_{\alpha,\beta}$ is divisible by $\prod\limits_{i:\ v_i\in \alpha\setminus\beta } (1-v_i).$
Define $$\delta: e_\sigma Ae_\tau\to e_\sigma Ae_{\sigma\cap\tau}\underset{A_{\sigma\cap\tau}}{\otimes}e_{\sigma\cap\tau}Ae_\tau$$ by the rule $\delta(nE_{\alpha,\beta})=nE_{\alpha,\beta\cap\alpha}\otimes E_{\beta\cap\alpha,\beta},$ where $n\in k[N]$. Note that since $nE_{\alpha,\beta}$ is an element of $e_\sigma A e_\tau$, $n$ must be divisible by $\prod\limits_{i:\ v_i\in \alpha\setminus\beta } (1-v_i),$ and hence $nE_{\alpha,\beta\cap\alpha}$ is an element of $e_\sigma A e_{\sigma\cap\tau}$. Similarly, $E_{\beta\cap\alpha,\beta}$ is an element of $e_{\sigma\cap\tau}Ae_\tau,$ and hence $\delta$ is well-defined.

It is clear that $\mu\delta=\id.$ To show that $\delta$ and $\mu$ are inverse isomorphisms it remains to show that for all $\alpha,\beta,\gamma$ that are faces of $\sigma,\tau,$ and $\sigma\cap\tau$ respectively, and $n,m\in k[N]$ such that $nE_{\alpha,\gamma}\in e_\sigma Ae_{\sigma\cap\tau},$ $mE_{\gamma,\beta}\in e_{\sigma\cap\tau} Ae_\sigma,$ we have

$$nE_{\alpha,\gamma}\otimes mE_{\gamma,\beta}=\delta\mu(nE_{\alpha,\gamma}\otimes mE_{\gamma,\beta})=nmE_{\alpha,\beta\cap\alpha}\otimes E_{\beta\cap\alpha,\beta}.$$ This follows form the fact that $mE_{\gamma,\beta\cap\gamma}\in A_{\sigma\cap\tau}$, which is used in the second step in the following computation: \begin{alignat*}{3}
   &\phantom{{}={}}nE_{\alpha,\gamma}\otimes mE_{\gamma,\beta}&&=nE_{\alpha,\gamma}\otimes mE_{\gamma,\beta\cap\gamma}E_{\beta\cap\gamma,\beta}&&=\\
   &=nmE_{\alpha,\gamma}E_{\gamma,\beta\cap\gamma}\otimes E_{\beta\cap\gamma,\beta}&&=nmE_{\alpha,\beta\cap\gamma}\otimes E_{\beta\cap\gamma,\beta}&&=\\  &=nmE_{\alpha,\beta\cap\gamma\cap\alpha}E_{\beta\cap\gamma\cap\alpha,\beta\cap\gamma}\otimes E_{\beta\cap\gamma,\beta}&&=nmE_{\alpha,\beta\cap\gamma\cap\alpha}\otimes E_{\beta\cap\gamma\cap\alpha,\beta\cap\gamma} E_{\beta\cap\gamma,\beta}&&=\\
   &=nmE_{\alpha,\beta\cap\gamma\cap\alpha}\otimes E_{\beta\cap\gamma\cap\alpha,\beta}
   &&=nmE_{\alpha,\beta\cap\gamma\cap\alpha}\otimes E_{\beta\cap\gamma\cap\alpha,\beta\cap\alpha}E_{\beta\cap\alpha,\beta}&&=
   \\&=nmE_{\alpha,\beta\cap\gamma\cap\alpha}E_{\beta\cap\gamma\cap\alpha,\beta\cap\alpha}\otimes E_{\beta\cap\alpha,\beta}&&=nmE_{\alpha,\beta\cap\alpha}\otimes E_{\beta\cap\alpha,\beta}.
\end{alignat*}
\end{proof}

Now for a regular fan $\Delta$ we define the category of descent data $\Desc(\Delta)$. An object of $\Desc(\Delta)$ is a collection $((M_\sigma)_{\sigma\in\Delta},(\phi_\sigma^\tau)_{\tau,\sigma\in\Delta})$, where $M_\sigma$ is a $A_\sigma$-module and $$\phi_\sigma^\tau: e_{\sigma\cap\tau} M_\sigma\xrightarrow[]{\sim} e_{\sigma\cap\tau}M_\tau$$ are isomorphisms satisfying the cocycle condition $$e_{\sigma\cap\tau\cap\omega}\phi_\sigma^\omega=e_{\sigma\cap\tau\cap\omega}\phi_\tau^\omega\circ e_{\sigma\cap\tau\cap\omega}\phi_\sigma^\tau.$$

A morphism $f: ((M_\sigma), (\phi^\tau_\sigma))\to ((N_\sigma),(\psi^\tau_\sigma))$ in $\Desc(\Delta)$ is a collection of morphisms of $A_\sigma$-modules $f_\sigma: M_\sigma\to N_\sigma$ such that the diagrams

\begin{equation}\label{def_descent_morphisms}
    \begin{tikzcd}
	{e_{\sigma\cap\tau}M_\sigma} && {e_{\sigma\cap\tau}N_\sigma} \\
	\\
	{e_{\sigma\cap\tau}M_\tau} && {e_{\sigma\cap\tau}M_\tau}
	\arrow["{e_{\sigma\cap\tau}f_\tau}"', from=3-1, to=3-3]
	\arrow["{e_{\sigma\cap\tau}f_\sigma}", from=1-1, to=1-3]
	\arrow["{\phi^\tau_\sigma}"', from=1-1, to=3-1]
	\arrow["{\psi_\sigma^\tau}", from=1-3, to=3-3]
\end{tikzcd}
\end{equation} for varying $\sigma$ and $\delta$ commute. Let $\Desc(\Delta)^{\mathrm{fd}}$ be the full subcategory of $\Desc(\Delta)$ with objects such that all $M_\sigma$ are finite-dimensional.

\begin{Proposition}\label{Prop_descent_for_A-mod}
    The categories $A(\Delta)\text{-}\mathrm{mod}$ and $\Desc(\Delta)$, and also $A(\Delta)\text{-}\mathrm{mod^{fd}}$ and $\Desc(\Delta)^{\mathrm{fd}}$, are equivalent. 
\end{Proposition}

\begin{proof}
    We have a bifibration $P:\mathcal{E}\to\mathcal{B}$ satisfying the Beck-Chevalley condition at $p: \{e_\sigma\}_{\sigma\in\Delta}\to \ast$ by \autoref{base_change}. The ideal $(e_\sigma)_{\sigma\in\Delta}$ is the whole ring $A(\Delta)$, so by \autoref{monadicity} the functor $p^*$ is also monadic. The Bénabou-Roubaud theorem then asserts the existence of an equivalence of a category of descent data  and the category of algebras over the monad $p^*p_!$, which is in turn equivalent to $A(\Delta)\text{-}\mathrm{mod}$ by the monadicity of $p^*$.  A module $M$ over $A(\Delta)$ is finite-dimensional if and only if $p^*M$ is, so this equivalence restricts to an equivalence between $A(\Delta)\text{-}\mathrm{mod^{fd}}$ and $\Desc(\Delta)^{\mathrm{fd}}$.
    
\end{proof}

\begin{Remark}
    The definition of the category of descent data used above should be familiar to algebraic geometers. In the category theory literature (for example, in \cite{janelidze1994facets}), the definition of descent data in the Bénabou-Roubaud theorem is slightly different, and more strongly resembles the definition of an object with the action of a monad. For the proof of the Bénabou-Roubaud and the connection between the two definition of descent data, see \cite{kahn2024benabouroubaud}.

\end{Remark}
\subsection{Equivalence of descent data}

Now that we have a descent theory for modules over $A(\Delta)$, we may proceed to show that the descent data in $\Desc(\Delta)$ are equivalent to descent data for perverse sheaves for the cover by the toric affine charts.

\begin{Proposition}\label{Prop_equiv_affine}
    Let $\sigma$ be a regular cone in $\mathbb{Q}^n$. For the corresponding smooth affine toric variety $X_{\sigma}$ there is an equivalence of categories $\Perv(X_{\sigma},k)$ and $A(\sigma)\text{-}\mathrm{mod^{fd}}.$
\end{Proposition}
\begin{proof}
    Fix a basis $\varepsilon_1,\ldots,\varepsilon_n$ of $\mathbb{Q}^n$ consisting of lattice elements. First, let's deal with the case when $\sigma$ is generated by the first $k$ vectors of the basis $\varepsilon_1,\ldots,\varepsilon_k.$ The corresponding fan can be viewed as the product of fans $\Delta_1,\ldots,\Delta_n$, where we consider $\Delta_i$ as being embedded into the linear span of $\varepsilon_i$ and consisting of two cones, the zero cone and the ray generated by $\varepsilon_i$, for $i\leqslant k$, and consisting of just the zero cone for $i> k$.
    
    The affine variety $X_{\sigma}$ is then isomorphic to the product $\prod\limits_{i=1}^n X_{\Delta_i}$, with each factor being either $\mathbb{C}$ stratified by 0 and its complement $\mathbb{C}^\times$, or simply $\mathbb{C}^\times$. We know what the categories of perverse sheaves on each factor are, and by a result of Lyubashenko \cite{lyubashenko2001exterior}, the category of perverse sheaves on a product of varieties with the stratification by products of strata may be identified with the Deligne tensor product of the categories of perverse sheaves on factors:
        $$ \Perv_\Sigma(X,k)\boxtimes \Perv_\Theta(Y,k)\xrightarrow[]{\sim}\Perv_{\Sigma\times\Theta}(X\times Y,k),$$ with the equivalence induced by universal property from the external tensor product functor $-\boxtimes -$.  For Deligne tensor products and their existence for categories with possibly infinitely many simple objects we refer to \cite{franco2013tensor}. For a pair of $k$-algebras $A,B$ the Deligne tensor product $A\text{-}\mathrm{mod^{fd}}\boxtimes B\text{-}\mathrm{mod^{fd}}$ may be identified with the category $A\otimes_k B\text{-}\mathrm{mod^{fd}}$. In our situation, this gives us an identification 
        $$\Perv(X_{\sigma},k)\cong A(\Delta_1)\otimes_k\ldots\otimes_k A(\Delta_n)\text{-}\mathrm{mod^{fd}}.$$
        If we think of all the algebras on the right side as of matrix algebras, then the right side is identified with  $A(\sigma)$ by the Kronecker product.

Alternatively, in the previous step one could avoid talking about Deligne tensor product and instead use the description of perverse sheaves on $\mathbb{C}^k\times(\mathbb{C}^\times)^{n-k}$ of Granger, Galligo and Maisonobe \cite{galligo1985mathcal}, identify it with modules over an algebra and explicitly construct an isomorphism with the algebra $A(\sigma)$ similarly to the case of perverse sheaves on $\mathbb{C}$ in \autoref{Example_C1}.

For an arbitrary nonsingular cone $\sigma$, fix a set of its integral generators $v_1,\ldots,v_k$. Since the cone is nonsingular, this set can be completed to a basis $v_1,\ldots,v_n$. The automorphism $\beta_\sigma$ of the lattice $N$ mapping $v_i$ to $\varepsilon_i$ maps the cone $\sigma$ onto the cone spanned by the first $k$ vectors $\varepsilon_i,\ldots, \varepsilon_k$, and hence induces an isomorphism $\phi_\sigma:\ X_{\sigma} \xrightarrow{\sim} X_{\Lambda_k},$ where $\Lambda_k$ is the fan consisting of the cone generated by the first $k$ basis vectors and its faces. It induces an equivalence of categories $\phi_{\sigma*}: \Perv(X_\sigma,k)\to\Perv(X_{\Lambda_k},k).$ The latter category is equivalent to $A(\Lambda_k)\text{-}\mathrm{mod^{fd}}$ by the equivalence constructed above.

As matrix algebras over $k[N]$, $A(\Lambda_k)$ consists of matrices $M$ with rows and columns indexed by cones in $\Lambda_k$ such that each component $M_{\omega,\tau}$ is divisible by $\prod\limits_{i:\ \varepsilon_i\in\omega\setminus\tau}(1-\varepsilon_i),$ while $A(\sigma)$ consists of matrices $Q$ with rows and columns indexed by cones in $\sigma$ such that each component $Q_{\omega,\tau}$ is divisible by $\prod\limits_{i:\ v_i\in\omega\setminus\tau}(1-v_i).$ We can use $\beta_\sigma$ to define an isomorphism $\gamma_\sigma: A(\sigma)\to A(\Lambda_k)$. Explicitly, we apply the extension of $\beta_\sigma$ to $k[N]$ to each matrix component $M_{\omega,\tau}$ and placing it into row $\beta_{\sigma}(\omega)$ and column $\beta_{\sigma}(\tau)$.

All of the above combined gives us the desired equivalence

$$\Perv(X_{\sigma},k)\xrightarrow{\sim}\Perv(X_{\Lambda_k},k)\xrightarrow{\sim}A(\Lambda_k)\text{-}\mathrm{mod^{fd}}\xrightarrow{\sim}A(\sigma)\text{-}\mathrm{mod^{fd}}.$$

\end{proof}

The equivalence in the previous proof involves some choices, in particular, the choice of the completion of the set of generators of a nonsingular cone to a basis. We will make sure that regardless of the choices made, such equivalences will behave nicely with regards to descent.

\begin{proof}[Proof of \autoref{Theorem_perverse_modules}]
    Our aim is to construct an equivalence $\Perv(X_\Delta)\cong A(\Delta)\text{-}\mathrm{mod^{fd}}.$ The latter category is equivalent to the category of descent data $\Desc(\Delta)$ by \autoref{Prop_descent_for_A-mod}. The category $\Perv(X_\Delta)$, by smooth descent for perverse sheaves, is also equivalent to the category of descent data with respect to the cover by affine toric varieties, which we denote by $\Desc^p(\Delta).$ Its objects are collections of perverse sheaves on the elements of the cover with gluing isomorphisms satisfying the cocycle condition. On each affine open corresponding to a cone $\sigma$, we have already constructed an equivalence between $\Perv_\Sigma(X_{\sigma})$ and $A(\sigma)\text{-}\mathrm{mod^{fd}},$ and to extend it to an equivalence $\Desc(\Delta)\xrightarrow{\sim}\Desc^p(\Delta)$ it remains to establish that these equivalences on open charts commute with the restriction functors.

For each cone $\sigma$ of $\Delta$ fix a basis containing the set of integral generators of the cone, a lattice automorphism $\beta_\sigma$ mapping the cone to the cone generated by the coordinate vectors, and use this automorphism to identify perverse sheaves over the corresponding chart with modules by the construction of \autoref{Prop_equiv_affine}. Let $\tau$ and $\sigma$ be cones of dimensions $s$ and $k$ respectively such that $\tau\prec\sigma$ (in particular, $\tau$ may equal $\sigma$).   Let $j_{\tau,\sigma}$ be the inclusion of $X_{\tau}$ into $X_{\sigma}$. Let us show that there are functors $F$ and $G$ such that the following diagram, with the rows as in \autoref{Prop_equiv_affine}, commutes
\[\begin{tikzcd}
	{\Perv(X_{\bar{\sigma}})} && {\Perv(X_{\Lambda_k})} && {A(\Lambda_k)\text{-}\mathrm{mod^{fd}}} && {A(\bar{\sigma})\text{-}\mathrm{mod^{fd}}} \\
	\\
	{\Perv(X_{\bar{\tau}})} && {\Perv(X_{\Lambda_s})} && {A(\Lambda_s)\text{-}\mathrm{mod^{fd}}} && {A(\bar{\tau})\text{-}\mathrm{mod^{fd}}}
	\arrow["\sim", from=1-1, to=1-3]
	\arrow["{j_{\tau,\sigma}^*}", from=1-1, to=3-1]
	\arrow["\sim", from=1-3, to=1-5]
	\arrow["F", from=1-3, to=3-3]
	\arrow["\sim", from=1-5, to=1-7]
	\arrow["G", from=1-5, to=3-5]
	\arrow["{e_\tau(-)}", from=1-7, to=3-7]
	\arrow["\sim", from=3-1, to=3-3]
	\arrow["\sim", from=3-3, to=3-5]
	\arrow["\sim", from=3-5, to=3-7]
\end{tikzcd}\]

 For the left square to commute, $F$ must be the inverse image functor $f_{\tau,\sigma}^*$ for the map $f_{\tau,\sigma}: X_{\Lambda_s}\to X_{\Lambda_k}$ being induced by the automorphism $\beta_\sigma \beta_\tau^{-1}$ of the lattice $N$. In the right square, we may view the restriction functor $e_\tau(-)$ as the inverse image functor for the inclusion morphism $A(\tau)\cong e_\tau A(\sigma)e_\tau\hookrightarrow A(\sigma)$. For the right square to commute, $G$ must then be the inverse image functor $g_{\tau,\sigma}^*$ with the map $g_{\tau,\sigma}: A(\Lambda_s)\to A(\Lambda_k)$ equal to the composite  

$$A(\Lambda_s)\xrightarrow{\gamma_\tau^{-1}}A(\tau)\hookrightarrow A(\sigma)\xrightarrow{\gamma_\sigma} A(\Lambda_k).$$  Let $e_{s,k}$ be the idempotent in $A(\Lambda_k)$ corresponding to the inclusion of fans $\Lambda_s\hookrightarrow \Lambda_k$.Possibly up to reordering the rows and columns, we may then rewrite this composite as $$A(\Lambda_s)\xrightarrow{\sim}A(\Lambda_s)\hookrightarrow A(\Lambda_k),$$ and, up to the same reordering of the coordinates, rewrite the morphism $f_{\tau,\sigma}$ as the composite 
$$X_{\Lambda_s}\xrightarrow{\sim}X_{\Lambda_s}\hookrightarrow X_{\Lambda_k},$$ with both of the left arrows induced by $\beta_\sigma \beta_\tau^{-1}$ as in \autoref{Prop_equiv_affine}

Writing $F$ and $G$ as composites allows us to reduce the commutativity of the resulting middle square to the commutativity of the two squares below, one corresponding to restriction to an open toric chart of $X_{\Lambda_k}$, the other to then pulling back along an automorphism:

\[\begin{tikzcd}
	{\Perv(X_{\Lambda_k})} && {A(\Lambda_k)\text{-}\mathrm{mod^{fd}}} && {\Perv(X_{\Lambda_s})} && {A(\Lambda_s)\text{-}\mathrm{mod^{fd}}} \\
	\\
	{\Perv(X_{\Lambda_s})} && {A(\Lambda_s)\text{-}\mathrm{mod^{fd}}} && {\Perv(X_{\Lambda_s})} && {A(\Lambda_s)\text{-}\mathrm{mod^{fd}}}
	\arrow["\sim", from=1-1, to=1-3]
	\arrow["{j_{s,k}^*}"', from=1-1, to=3-1]
	\arrow["{e_{s,k}(-)}", from=1-3, to=3-3]
	\arrow["\sim", from=1-5, to=1-7]
	\arrow["{a_{\tau,\sigma}^*}"', from=1-5, to=3-5]
	\arrow["{b_{\tau,\sigma}^*}", from=1-7, to=3-7]
	\arrow["\sim", from=3-1, to=3-3]
	\arrow["\sim", from=3-5, to=3-7]
\end{tikzcd}\]

The commutativity of the left square is evident, because we basically have already seen that this diagram commutes in the case when $X_{\Lambda_k}=\mathbb{C}$ and $X_{\Lambda_k}=\mathbb{C}^\times$, and the more general case follows from this by taking products. In the right square, the two vertical arrows are inverse image functors along an automorphism of $X_{\Lambda_s}$ and an automorphism of $A(\Lambda_s)$, both induced by the same automorphism $\beta_\sigma \beta_\tau^{-1}$ of $N$. This square commutes, because the top arrow in the square uses $\beta_\sigma$ for its construction, the bottom arrow uses $\beta_\tau$, so they differ precisely by a pullback along $\beta_\sigma \beta_\tau^{-1}$.

\end{proof} 

\section{Equivariant perverse sheaves}

In the previous section we identified the category of perverse sheaves on a toric variety $X_\Delta$ with the category of finite-dimensional modules over an algebra $A(\Delta)$. The center of the algebra $A(\Delta)$ is naturally identified with the group algebra $k[\pi_1(T)].$ Let $G$ be a closed algebraic subgroup of $T$. Our goal now is to obtain an algebraic description of the category $\Perv_G(X_\Delta)$ of $G$-equivariant perverse sheaves (but which are still constructible with respect to the stratification by $T$-orbits).

With this in mind, let us consider the short exact sequence 

$$1\to G\to T\to T/G\to 1.$$ The algebraic group $T/G$ is also an algebraic torus of dimension $\dim T-\dim G.$ There is an induced map $k[\pi_1(T)]\to k[\pi_1(T/G)]$ which we use to define $$A_G(\Delta)=A(\Delta)\underset{k[\pi_1(T)]}{\otimes}k[\pi_1(T/G)].$$

The claim is that the data of a $G$-equivariant perverse sheaf on $X_\Delta$ is equivalent to the data of a finite-dimensional module over $A_G(\Delta)$. Let us briefly recall the constructions of $\Phi_f$ and $\Psi_f$, following Kashiwara and Schapira \cite{kashiwara2002sheaves}.

Let $f: X\to \mathbb{C}$ be a holomorphic map and $p:\mathbb{C}\to\mathbb{C}$ the exponential map realizing $\mathbb{C}$ as the universal cover of $\mathbb{C}^\times$. Let $\mathrm{tr}$ be the trace morphism $p_!k_{\mathbb{C}}\cong p_!p^!k_{\mathbb{C}}\to k_{\mathbb{C}}$ and define $K$ as the complex

$$0\to p_!p^!k_{\mathbb{C}}\xrightarrow{\mathrm{tr}}k_{\mathbb{C}}\to 0$$ with $k_{\mathbb{C}}$ in degree $0$. Then, by definition (up to a shift), we have

$$\Psi_f(F)=i^*R\sheafhom(f^*p_!k_{\C},F),\quad \Phi_f(F)=i^*R\sheafhom(f^*K,F).$$ The monodromy operators are then defined by the action of the endomporphism $z\mapsto z+2\pi i$ of $\C$  on the sheaf $k_\C$, which induces automoprhisms of both $p_! k_\C$ and $K$.

Before starting the proof, let us deal with a few examples

\begin{Example}\label{Example_discrete}
    Let $X_\Delta=\mathbb{C}$ with the coordinate function $z$ and $G=\mathbb{Z}/p\mathbb{Z}\subset \mathbb{C}^\times$. Let $\xi=e^{\frac{2\pi i}{p}}$ be the generator of $G$.
The spaces $\Psi$ and $\Phi$ in the quiver description of perverse sheaves are obtained by applying the vanishing cycle and nearby cycles functors $\Phi_f$ and $\Psi_f$ for $f=z$. Let $t$ denote the monodromy automorphism defined by the maps $\id_\Psi-uv$ and $\id_\Phi-vu$. 

There are the action and projection maps $a,\ pr_2: G\times X_\Delta\to X_\Delta$, and the equivariant structure on a perverse sheaf $F$ is given by an isomorphism of perverse sheaves (shifted by $[\dim G]$) $\theta: pr_2^* F\to a^* F$. The space $G\times X_\Delta$ in this case consists of $p$ disjoint copies of $X_\Delta$ indexed by the elements of $G$, and the additional axioms that must be satisfied by $\theta$ imply that $\theta$ is uniquely determined by its restriction to the component corresponding to the generator $\xi$ of $G$. Denote this restriction by $\theta_\xi: F\to m_\xi^* F$, where $m_\xi$ is the multiplication by $\xi$.

If we were working with equivariant local systems on $\mathbb{C}^\times$, the composition of going $1/p$-th of the way around 0 and the isomorphism given by the equivariant structure would give us a $p$-th root of the monodromy $t$. We shall see that the same thing happens with perverse sheaves.

According to Exercise 8.15 of \cite{kashiwara2002sheaves}, $$\Psi_z(m_\xi^* F)\cong\Psi_z((m_{\xi^{-1}})_* F)\cong (m_{\xi^{-1}})_*\Psi_{\xi^{-1}z}(F)\cong \Psi_{\xi^{-1}z}(F),$$ where the last isomorphism holds because multiplication acts trivially on 0, the support of the nearby cycles sheaf, and similarly $\Phi_z(m_\xi^* F)\cong \Phi_{\xi^{-1}z}(F)$. Explicitly, we have $\Psi_{\xi^{-1}z}(F)\cong i^*R\sheafhom(f^*p_!\tau(\frac{-2\pi i}{p})_!k_{\C},F)$, where $\tau(\frac{-2\pi i}{p})$ is the map $z\mapsto z-\frac{-2\pi i}{p}$. By identifying constant sheaves on $\C$ with their stalks at 0 and considering the map $(k_\C)_{0}\to(k_\C)_{\frac{2pi}{p}}$ given by moving along the line segment $[0,\frac{2pi}{p}]$, we get a map of sheaves $k_{\C}\to \tau(\frac{-2\pi i}{p})_!k_\C.$ By the contravariance of $\sheafhom$ in the first argument, we get an isomorphism $\gamma_\Psi: \Psi_{\xi^{-1}z}(F)\to \Psi_{z}(F)$. In the same manner we get an isomorphism $\gamma_\Phi: \Phi_{\xi^{-1}z}(F)\to \Phi_{z}(F)$.

By functoriality, these two morphisms will commute with the canonical maps between nearby and vanishing cycles, and hence give the isomorphism  $\gamma: m_\xi^* F\to F.$ Unwinding the axioms of equivariant structures further, we find that $\theta_\xi$ defines an equivariant structure on $F$ precisely when $(\gamma\circ \theta_\xi)^p=t.$ Therefore, an equivariant structure on a perverse sheaf in this case is just a choice of a $p$-th root of the monodromy.

After choosing coordinates, the map $T\to T/G$ is just the map $z\to z^p$, and the induced map $k[\pi_1(T)\to k[\pi_1(T/G)]$ is just the map $k[t,t^{-1}]\to k[s,s^{-1}]$ taking $t$ to $s^p$. Therefore,

$$\Perv_G(X_\Delta)\cong A_G(\Delta)\text{-}\mathrm{mod}^{\mathrm{fd}}=A(\Delta)\underset{k[t,t^{-1}]}{\otimes}k[s,s^{-1}]\text{-}\mathrm{mod}^{\mathrm{fd}}.$$
\end{Example}
\begin{Example}\label{Example_connected}

Now we consider the case $X_\Delta=\mathbb{C}$ and $G=\mathbb{C}^\times.$ The space $G\times X_\Delta$ is also a smooth toric variety, and a perverse sheaf $H$ on $G\times X_\Delta$ is given by a diagram 

\[\begin{tikzcd}
	\Psi && \Phi
	\arrow["v", curve={height=-18pt}, from=1-3, to=1-1]
	\arrow["u", curve={height=-18pt}, from=1-1, to=1-3]
\end{tikzcd}\] with an additional automorphism $s$ corresponding to traveling along the loop around 0 in $G$. For $H=pr_2^* F[1]$, this automorphism is clearly the identity. The claim is that for $H=a^* F[1]$, this automorphism is exactly $t$. This is easy for the component of $s$ acting on $\Psi(H)$, since in this case, taken with the monodromy action $t$, it is just the local system $F|_{\mathbb{\C}^\times}$. For the component of $s$ acting on $\Phi(H)$, let $u,z$ be the coordinates on $\C^\times \times \C$ and use the formula $\Phi_z(H)=(R\Gamma_{\{\mathrm{Re }\ z\geqslant0\}}H)|_{z=0}$ (Exercise 8.13 of  \cite{kashiwara2002sheaves}). The stalk of this at the point $(1,1)$ is clearly isomorphic to $\Phi_z(F)$, while going around the loop around $0$ in $G=\C^\times$ corresponds to turning the halfplane $\mathrm{Re }\ z\geqslant 0$ once around the origin in the base space, which gives precisely the action of $t$ on the vanishing cycles of the original perverse sheaf $F$.

All in all, the isomorphism $\theta$ is the same as an automorphism, also denoted $\theta$, of the above diagram such that $\theta=t\theta$. Such an automorphism exists if and only if $t=\id$. We find that a $\mathbb{C}^\times$-equivariant perverse sheaf on $\mathbb{C}$ is given by a diagram \[\begin{tikzcd}
	\Psi && \Phi
	\arrow["v", curve={height=-18pt}, from=1-3, to=1-1]
	\arrow["u", curve={height=-18pt}, from=1-1, to=1-3]
\end{tikzcd}\] such that $uv=0$ and $vu=0$. 

The quotient torus $T/G$ is the trivial group, and we find $$A_G(\Delta)=A(\Delta)\underset{k[t,t^{-1}]}{\otimes}k,$$ as expected.
\end{Example}
\begin{Theorem}\label{Theorem_equivariant_perverse_modules}
    The category $\Perv_G(X_\Delta)$ is equivalent to the category $A_G(\Delta)-\mathrm{mod^{fd}}$ of finite-dimensional left $A_G(\Delta)$-modules.
\end{Theorem}
\begin{proof}
    Consider the sequence $$0\to G\to T\to T/G\to 0.$$ There exists a choice of coordinates $z_1,\ldots, z_n$ on $T$ and $u_1,\ldots,u_m$ on $T/G$ such that the second map has the form $z_i\mapsto u_i^{d_i}$, $d_i\geqslant1$, for $i\leqslant m$, and $t_i\mapsto 1$ for $i>m$. In these coordinates, $G$ is the product of discrete groups $\mathbb{Z}/d_i\mathbb{Z}$ generated by $\xi_i=e^{\frac{2\pi i}{d_i}}$ and embedded into the $i-$th factor of $T=(\mathbb{C}^\times)^n$ and a connected group $G_0=(\mathbb{C}^\times)^{n-m}$ embedded into $T$ as the last $(n-m)$ factors. Let $t_i$ be the loop around 0 in the $i$-th factor of $T$.

    We can cover the smooth variety $X_\Delta$ by affine toric charts and consider the equivariant structure on each such chart. For the time being, we assume that $X=X_\Delta$ is affine. The description of a perverse sheaf on $X$ as a diagram of vector spaces is obtained by repeatedly applying the functors $\Psi_f$ and $\Phi_f$ for various $f$ which cut out the torus-invariant divisors. In fact, the result will be independent of the order of these functors, which may be seen by from identifying $X$ with a product of copies of $\C$ and $\C^\times$ and, using the result of \cite{lyubashenko2001exterior}, identifying the category of perverse sheaves on the affine chart with the Deligne product of such categories on each factor. 

    Now a $G$-equivariant structure is a specific isomorphism of perverse sheaves on $G\times X$, which is a disjoint union of affine varieties. By essentially the same computation as in \autoref{Example_connected} above, for a $G$-equivariant structure on a perverse sheaf $F$ to exist it is necessary for the fundamental group of $G_0$, i.e. the automorphisms $t_i$ for $i>m$, to act trivially on it. Assuming it does indeed act trivially, an equivariant structure is determined by the isomorphisms $\theta_i: F\to m_{\xi_i}^* F$, where $\xi_i$ is viewed as an element of $T$ through the embedding $\mathbb{Z}/d_i\mathbb{Z}\subset T$.
    A repeated application of the formula $$\Psi_f(m_{\xi_i}^* F)\cong m_{\xi_i}^* \Psi_{f\circ m_{\xi_i}^{-1}}(F)$$ and its analogue for $\Phi_f$ essentially reduces the study of this equivariant structure to the one-dimensional cases described in \autoref{Example_discrete}, and we find that giving an equivariant structure is equivalent to giving a $d_i$-th root $s_i$ of each $t_i$ for $i\leqslant m$.

    Finally, for the $G$-equivariant structures on the restrictions of a perverse sheaf to affine charts to glue to a global $G$-equivariant structure we must only require for the chosen roots of $t_i$ to glue. All in all, giving a $G$-equivariant structure on a (possibly non-affine) $X_\Delta$ is equivalent to requiring that the automorphisms $t_i$ for $i>m$ act trivially on it and giving a $d_i$-th root for the action of each $t_i$ for $i\leqslant m$, and this is exactly the structure of a module over $A_G(\Delta).$
 
\end{proof}

In particular, this theorem describes the category of perverse sheaves on a toric DM stack in the sense of \cite{borisov2005toricDM}.

\begin{Remark}
    The Cox construction realizes a smooth toric variety $X$ as a quotient of a quasi-affine toric variety $U$ by a subgroup of the torus $G$. Then we have $$\Perv(X)\cong \Perv_G(U),$$ so the previous theorem allows us to recover the results of the previous section as long as  we first describe the category of perverse sheaves on the variety $U$. But $U$ is an open union of strata of ${\mathbb{C}}^n$, and we know what perverse sheaves look like on ${\mathbb{C}}^n$ in terms of diagrams of vector spaces indexed by the strata. To see that the perverse sheaves on $U$ simply correspond to a subdiagram of spaces indexed by the strata of $U$, one could either use a simplified version of the argument in the previous section or work this result out from the proof of, say, Maisonobe, Galligo and Granger \cite{galligo1985mathcal}.

\end{Remark}

\section{Beilinson's Theorem for realizable stratifications}

The category $\Perv(X)$ of perverse sheaves that are constructible with respect to some (not fixed a priori) stratification is defined as a heart of a $t$-structure on $D^{b}_{\mathrm{cons}}(X)$, the full subcategory of the derived category of sheaves $D^{b}(\mathrm{Sh}(X))$ with objects having constructible cohomology sheaves. This triangulated category admits what is called a filtered version \cite{beilinson2006derived}. There is a triangulated realization functor \[\mathrm{real:\ } D^b(\Perv(X)) \to D^{b}_{\mathrm{cons}}(X),\] which is $t$-exact with respect to the standard $t$-structure on the left and the perverse $t$-structure on the right, and
whose restriction to $\Perv(X)$ is the inclusion functor $\Perv(X)\hookrightarrow D^{b}_{\mathrm{cons}}(X)$ \cite{BBD1982,achar2021perverse}. A theorem of Beilinson states that the functor $\mathrm{real}$ is an equivalence \cite{beilinson2006derived}. For a fixed stratification $\Sigma$, however, the functor \[\mathrm{real:\ } D^b(\Perv_\Sigma(X) ) \to D^{b}_{\Sigma}(X)\] is rarely an equivalence.

Let $\mathrm{Loc}(X)$ and $\mathrm{Loc^{ft}}(X)$ denote the categories of all local systems and of local systems of finite type respectively and $D^b_{\mathrm{loc}}(X)$ and $D^b_{\mathrm{locf}}(X)$ the full subcategories of $D^{b}(\mathrm{Sh}(X))$ with objects with cohomology sheaves which are locally constant or locally constant of finite type respectively. 

Consider the case when $\Sigma$ is the trivial stratification, which consists of a single stratum $X$.  Both $\Perv_\Sigma(X)$ and $\Cons_\Sigma(X)$, the category of sheaves constructible with respect to $\Sigma$, coincide (up to a shift in the perverse case) with the category $\mathrm{Loc^{ft}}(X)$. Then a classic result states that the functor $\mathrm{real}: D^b(\mathrm{Loc}(X))\to D^b_{\mathrm{loc}}(X)$ is an equivalence if and only if the higher homotopy groups $\pi_n(X,*)$, $n\geqslant 2$, are trivial \cite{achar2021perverse}. So already the topology of $X$, or of the strata for a nontrivial stratification $\Sigma$, may prevent the realization functor from being an equivalence. There is also an essentially representation theoretic issue: restricting to local systems of finite type rarely preserves the higher $\Ext$-functors, and the realization functor $\mathrm{real}: D^b(\mathrm{Loc^{ft}}(X))\to D^b_{\mathrm{locf}}(X)$ is usually not an equivalence even for $X$ of type $K(\pi,1)$. The aim in this section is to establish a sufficient condition on a stratification $\Sigma$ for the realization functor to be an equivalence.

Let $\mathcal{T}$ be a triangulated category, and let $\mathcal{C}\subset \mathcal{T}$ be the
heart of a $t$-structure. Let $X, Y \in \mathcal{C}$ , and let $n > 0$. A morphism $f : X \to Y [n]$ is
said to be effaceable if there are morphisms $p : X^{'} \to X$ and $i : Y \to Y^{'}$ in $\mathcal{C}$ such
that $p$ is surjective, $i$ is injective, and $i[n]\circ f\circ p : X^{'} \to Y^{'}[n]$ is the zero morphism.

 The question of whether the realization functor $\mathrm{real}: \mathcal{C}\to\mathcal{T}$ is an equivalence reduces to the question of whether morphisms between shifts of objects of the heart $\mathcal{C}$ are effaceable as morphisms of $\mathcal{T}$ \cite{achar2021perverse}.

The key observation is the following classic result:

\begin{Proposition}
    
\label{theorem_effaceability_equiv}
Let $\mathcal{T}$ be a triangulated category that admits a filtered version, and let $\mathcal{C}$ be the heart of a bounded $t$-structure on $\mathcal{T}$. The following conditions are equivalent:
\begin{enumerate}
    \item The functor \[\mathrm{real:\ } D^b(\mathcal{C} ) \to \mathcal{T}\]is an equivalence of categories.
    \item For all $X, Y \in \mathcal{C}$ and all $n > 0$, every morphism $X \to Y [n]$ in $\mathcal{T}$ is
effaceable.    
\end{enumerate}
\end{Proposition}

The main idea of the proof of Beilinson's theorem is to find a good smooth open set $U$ in a variety $X$ that is, at least locally, a complement of a divisor of a regular function, and on which the functor \[D^b\mathrm{Loc^{ft}}(U) \to D^b_{\mathrm{locf}}(U)\] is an equivalence, and then using the open-complement exact triangles and Noetherian induction. The following definition basically specifies a class of stratifications for which the strata themselves may be used in place of $U$.

\begin{Definition}
    A stratification $\Sigma$ of an affine variety $X$ is realizable if
\begin{enumerate}
    \item $\Sigma$ is good,
    \item Each stratum $X_s$, $s\in \Sigma$, is smooth, affine and connected,
    \item On each stratum $X_s$, $s\in \Sigma$, the realization functor \[\mathrm{real:\ } D^b\mathrm{Loc^{ft}}(X_s) \to D^b_{\mathrm{locf}}(X_s)\] is an equivalence,
    \item If $Z$ and $W$ are closed unions of strata of $X$ such that $Z\subset W$ and $W\setminus Z$ consists of a single stratum, then $Z=f^{-1}(0)$ for some regular function $f: W\to \mathbb{A}^1$.
\end{enumerate}

    A stratification $\Sigma$ of a variety $X$ is realizable if $X$ can be covered by finitely many open sets $U_i$ such that each $U_i$ is an affine union of strata of $\Sigma$ and the stratification induced on $U_i$ by $\Sigma$ is realizable. 

\end{Definition}

 We remind the reader that a stratification $\Sigma$ is called good if for an inclusion of a stratum $j:X_s\to X$ the direct image $j_*L$ of a local system $L$ on $X_s$ is constructible with respect to $\Sigma$. This is a technical condition, satisfied for, for example, stratifications satisfying the Whitney conditions, and is necessary for the category $\Perv_\Sigma(X)$ to be well-defined.

In particular, the orbit stratification of a toric variety is realizable. Indeed, such a stratification is good, with each stratum isomorphic to some $(\mathbb{C}^\times)^m$, so satisfies the first two parts of the definition. The realization functor for local systems of finite type on $(\mathbb{C}^\times)^m$ is an equivalence, because such a space is of type $K(\pi,1)$ with a finitely generated free abelian fundamental group (for a more geometric argument, see Chapter 4.5 of \cite{achar2021perverse}). Finally, this stratification may also be obtained by taking the torus-invariant divisors $D_i$ and stratifying the variety in such a way that points in one stratum either all belong to some $D_i$, or none of them belong to said $D_i$. Then a key fact is that while not any torus-invariant divisor on an affine toric variety is set-theoretically a divisor of a regular function, the union of all torus invariant divisors is. From this description and this fact, the final part of the above definition follows easily.

It turns out that the proof of Beilinson's theorem essentially goes through for such stratifications. To see this, we need a fact from the theory vanishing and nearby cycles.

\begin{Proposition}
   Let $i: Z\hookrightarrow X$ be the inclusion of a closed union of strata $Z$, where $Z = f^{-1}(0)$ for some regular function $f : X \to \mathbb{A}^1$. Then the functor
$i_* : D^b\mathrm{Perv}_\Sigma(Z) \to D^b\mathrm{Perv}_\Sigma(X)$
is fully faithful. Its image consists of chain complexes $F\in D^b\mathrm{Perv}_\Sigma(X, k)$ such that
for all $k, H^k(F)$ is supported on $Z$.
\end{Proposition}\label{Prop_supp_Z}
\begin{proof}
    This is a classical result. A possible reference is Proposition 4.5.7 of \cite{achar2021perverse}. To see that nothing breaks down when the stratification $\Sigma$ is fixed, one should make sure that the maximal extension functor $\Theta_f: \mathrm{Perv}(X\setminus Z)\to\mathrm{Perv}(X)$ restricts to a functor $\Theta_f: \mathrm{Perv}_\Sigma(X\setminus Z)\to\mathrm{Perv}_\Sigma(X)$, which is straightforward.
\end{proof}

The following proof closely follows the exposition of the proof of Beilinson's theorem in \cite{achar2021perverse}.
\begin{Theorem}
   Let $X$ be a variety with a realizable stratification $\Sigma$. Then the realization functor
\[\mathrm{real:\ } D^b\mathrm{Perv}_\Sigma(X) \to D^b_\Sigma(X)\]
is an equivalence of categories.
\end{Theorem}
\begin{proof}
    By \autoref{theorem_effaceability_equiv}, it is enough to show that for any two perverse sheaves $F, G \in \mathrm{Perv}_\Sigma(X, k)$ and any $k \geqslant 0$, the map
\[\mathrm{real:\ } \mathrm{Ext}^k_{\mathrm{Perv}_\Sigma(X)}(\F, \G) \to \mathrm{Hom}_{D^b_\Sigma(X)}(F, G[k])\] induced by $\mathrm{real}$ is an isomorphism. We proceed by noetherian induction.

\emph{Step 1. Assume $X$ is affine and $\mathrm{supp\ }\F\cup \mathrm{supp\ }\G$ is a proper closed subvariety of $X$.}

Since $\Sigma$ is realizable, $\mathrm{supp\ }\F\cup \mathrm{supp\ }\G$ is contained in some $Z=f^{-1}(0)$ for some regular function $f: X\to \mathbb{A}^1$. Let $i:Z\to X$ be the inclusion map. Since $\F$ and $\G$ are supported on $Z$, we have $\F\cong i_*\F^{'}$, $\G\cong i_*\G^{'}$ for some $\F^{'},\G^{'}\in\Perv_\Sigma(Z).$ Now consider the following diagram:

\[\begin{tikzcd}
\Ext_{\Perv_\Sigma(Z)}^k(\F^{'},\G^{'}) \arrow{r}{\mathrm{real}} \arrow[swap]{d}{i_*} & \Hom_{D^b_\Sigma(Z)} (\F^{'},\G^{'}[k])\arrow{d}{i_*} \\
\Ext_{\Perv_\Sigma(X)}^k(i_*\F^{'},i_*\G^{'}) \arrow{r}{\mathrm{real}} & \Hom_{D^b_\Sigma(X)} (i_*\F^{'},i_*\G^{'}[k])
\end{tikzcd}
\]

The top arrow is an isomorphism by induction, the left arrow is an isomorphism by Proposition \ref{Prop_supp_Z}, while the right arrow is an isomorphism because the functor $i_*$ is fully faithful. The square commutes, and hence the bottom arrow is an isomorphism.

\emph{Step 2. Assume $X$ is affine and assume $\mathrm{supp\ }\G$ is a proper closed subvariety of $X$.} 

It is easy to see by induction on the number of composition factors that it is enough to cover the case when $\F$ is simple. If $\mathrm{supp\ }\F\cup \mathrm{supp\ }\G$ is a proper closed subvariety, we are done by Step 1. Otherwise, $\mathrm{supp\ }\F$ contains an open stratum $U$ that does not meet $\mathrm{supp\ }\G$, and by definition of a realizable stratification, it is smooth and affine.  Let $j: U\hookrightarrow X$ be the inclusion map. We have that $j^*\G=0$ and $j^*\F$ a shifted local system. Then $j_!j^*\F$ is a perverse sheaf, since the morphism $j$ is affine. The adjunction map $j_!j^*\F\to \F$ is nonzero, and since $\F$ is simple, it must be an epimorphism. Let $K$ be its kernel. By construction, $K$ is supported on the complement of $U$, so  $\mathrm{supp\ }\K\cup \mathrm{supp\ }\G$ is a proper closed subset. Consider the commutative diagram 
\[
\adjustbox{scale=0.95,center}{
\begin{tikzcd}
\Ext^{k-1}(j_!j^*\F,\G) \arrow{r} \arrow[swap]{d} & \Ext^{k-1}(K,\G) \arrow{r} \arrow[swap]{d}&\Ext^k(\F,\G) \arrow{r} \arrow[swap]{d}&\Ext^k(j_!j^*\F,\G) \arrow[swap]{d}\\ \Hom(j_!j^*\F,\G[k-1]) \arrow{r} & \Hom(K,\G[k-1]) \arrow{r} &\Hom(\F,\G[k]) \arrow{r} &\Hom(j_!j^*\F,\G[k])
\end{tikzcd}
}
\]
where all the $\Ext$ functors are taken in $\Perv_\Sigma(X)$ and all the $\Hom$ functors are taken in $D^b_\Sigma(X)$. Since $j_!$ and $j^*$ are a pair of adjoint exact functors, we have 

\[\Ext^{m}(j_!j^*\F,\G)\cong \Ext^{m}(j^*\F,j^*\G)=0\] for any $m$, so the first and last terms in the first row vanish. A similar argument in the derived category of sheaves shows that the first and last terms in the second row also vanish. The arrow in the second column is an isomorphism by Step 1, and hence so is the arrow in the third column.

\emph{Step 3. Conclusion for affine $X$ and arbitrary $\F,\G$.}

Again, we may assume that $\G$ is simple. There is an open stratum $U$ with the embedding $j: U\to X$ such that both $j^*\F,j^*\G$ are shifted local systems. Since $j$ is affine and $\Sigma$ is a good stratification, $j_*j^*\G$ is a perverse sheaf, and the adjunction map $\G\to j_*j^*\G$ is nonzero. Since $\G$ is simple, it must be a monomorphism. Let $K$ be its cokernel, and consider the diagram\[
\adjustbox{scale=0.8,center}{
\begin{tikzcd}
\Ext^{k-1}(\F,j_*j^*\G) \arrow{r} \arrow[swap]{d} & \Ext^{k-1}(\F,K) \arrow{r} \arrow[swap]{d}&\Ext^k(\F,\G) \arrow{r} \arrow[swap]{d}&\Ext^k(\F,j_*j^*\G) \arrow[swap]{d} \arrow{r} & \Ext^{k}(\F,K) \arrow[swap]{d} \\ 
\Hom(\F,j_*j^*\G[k-1]) \arrow{r} & \Hom(\F,K[k-1]) \arrow{r} &\Hom(\F,\G[k]) \arrow{r} &\Hom(F,j_*j^*\G[k])\arrow{r} & \Hom(\F,K[k])
\end{tikzcd}
}
\]

Since $K$ is supported on the complement of $U$, the second and fifth vertical maps are isomorphisms by Step 2.
For the first and fourth vertical maps, consider the commutative diagram

\[\begin{tikzcd}
\Ext_{\Perv_\Sigma(U)}^k(j^*\F,j^*\G) \arrow{r} \arrow[swap]{d} & \Ext_{\Perv_\Sigma(X)}^k(\F,j_*j^*\G)\arrow{d}{i_*} \\ 
\Hom_{D^b_\Sigma(U)} (j^*\F,j^*\G[k])\arrow{r} & \Hom_{D^b_\Sigma(X)} (\F,j_*j^*\G[k])
\end{tikzcd}
\]

The horizontal arrows are isomorphisms coming from the adjunction of the exact functors $j_!$ and $j^*$. As for the left arrow, we have that $\Perv_\Sigma(U)$ is just $\mathrm{Loc^{ft}}(U)$ up to a shift, and similarly $D^b_\Sigma(U)$ is just $D^b_{\mathrm{locf}}(U)$, so the arrow is an isomorphism by property 3 of a realizable stratification. By commutativity, the right arrow is also an isomorphism. By the five lemma, the middle arrow in the large diagram is an isomorphism, which is exactly what we wanted to prove.

\emph{Step 4. Conclusion for arbitrary $X$ and arbitrary $\F,\G$.}
Let's use the following fact: if $\mathcal{T}=D^b(\mathcal{A})$ is a derived category of an abelian category $\mathcal{A}$, then for $n>0$ every morphism $f: X\to Y[n]$, $X,Y\in \mathcal{A}$ is left-effaceable, i.e. there exists an epimorphism $p: P\to X$ such that $f\circ p=0$. (See Lemma A.5.17 of \cite{achar2021perverse}).

The goal is to show that every morphsim $\varphi\in \Hom_{D^b_\Sigma(X)}(F,G[n])$ between perverse sheaves is effaceable for $n>0$. Let $U_1,\ldots, U_k$ be a finite affine open cover of $X$ by unions of strata with a realizable stratification. let $j_i: U_i\to X$ be the embeddings. Above we have shown that the categories $D^b_\Sigma(U_i)$ are equivalent to the derived categories $D^b(\Perv_\Sigma(U_i))$, so there are surjective maps $p_i: P_i\to j_i^*F$ such that $j_i^* \varphi\circ p_i=0.$ By adjunction, we get the maps $\bar{p}_i: j_{i!} P_i\to F$. Note that since $j_i$ is affine, $j_{i!}P_i$ is a perverse sheaf. The commutative diagram

\[\begin{tikzcd}
	{j_{i!}\mathcal{P}_i} &&& {j_{i!}j_i^*F} &&& {F} \\
	\\
	&&& {j_{i!}j_i^*G[k]} &&& {G[k]}
	\arrow["{j_{i!}(p_i)}", from=1-1, to=1-4]
	\arrow[from=1-4, to=1-7]
	\arrow["{\overline{p}_i}", curve={height=-24pt}, from=1-1, to=1-7]
	\arrow["{j_{i!}(j_i^*(\varphi)\circ p_i)=0}"'{pos=0.3}, curve={height=12pt}, from=1-1, to=3-4]
	\arrow["{j_{i!}j_i^*\varphi}"', from=1-4, to=3-4]
	\arrow[from=3-4, to=3-7]
	\arrow["\varphi", from=1-7, to=3-7]
\end{tikzcd}\]
shows that $\varphi\circ\overline{p}_i=0$. Summing over $i$, we obtain a map 

\[\overline{p}:\bigoplus_{i=1}^nj_{i!}\mathcal{P}_i\to F\]
such that $\varphi\circ\overline{p}=0$. The restriction $j^*_i\overline{p}$ is surjective for any $i$, so the cokernel of $\overline{p}$ vanishes when restricted to any $U_i$, and hence is simply zero. Therefore, $\overline{p}$ is surjective and effaces $\varphi$. 
\end{proof}

As an application, we get the following corollary for normal, not necessarily smooth, toric varieties

\begin{Corollary}
    Let $X$ be a toric variety with the orbit stratification $\Sigma$. Then
\[\mathrm{real:\ } D^b\mathrm{Perv}_\Sigma(X) \to D^b_\Sigma(X)\]
is an equivalence of categories.
\end{Corollary}
    
\begingroup
\let\itshape\upshape
\bibliography{ref.bib}
\bibliographystyle{abbrv}
\endgroup
\end{document}